\documentclass[11pt]{article}
\usepackage{amsmath,amssymb,amsthm,mathptmx,enumerate,url}
\usepackage[margin=1in]{geometry}

\newtheorem{thm}{Theorem}[section]
\newtheorem*{thm*}{Theorem}
\newtheorem{lemma}[thm]{Lemma}
\newtheorem*{lemma*}{Lemma}

\newtheorem*{prop*}{Proposition}
\newtheorem{cor}[thm]{Corollary}

\newtheorem{fact}[thm]{Fact}
\newtheorem*{not*}{Notation}
\newtheorem{claim}[thm]{Claim}
\newtheorem*{claim*}{Claim}
\newtheorem*{fact*}{Fact}

\newtheorem{conj}[thm]{Conjecture}

\newcommand{\Q}{\mathbb{Q}}
\newcommand{\Z}{\mathbb{Z}}
\newcommand{\T}{\mathbb{T}}
\newcommand{\R}{\mathbb{R}}
\newcommand{\C}{\mathbb{C}}

\DeclareMathOperator{\Ex}{\mathbb{E}}

\newcommand{\GL}{\mathrm{GL}}
\newcommand{\eps}{\varepsilon}
\newcommand{\supp}{\mathrm{Supp}}

\newcommand{\poly}{\mathrm{poly}}

\renewcommand{\hat}[1]{\widehat{#1}}
\newcommand{\choosesq}[2]{\genfrac{[}{]}{0pt}{}{#1}{#2}}
\newcommand{\ip}[1]{\langle #1 \rangle}
\renewcommand{\Re}{\mathrm{Re}}
\newcommand{\one}[1]{\mathbf{1}\{#1\}}

\newcommand{\ignore}[1]{}

\renewcommand{\Pr}{\mathbb{P}}

\begin{document}

\title{Probabilistic existence of rigid combinatorial structures
    \\ \begin{large} (extended abstract version) \end{large}}

\author{Greg Kuperberg \thanks{University of California, Davis.  E-mail:
    \texttt{greg@math.ucdavis.edu}. Supported by NSF grant CCF-1013079.}
    \and Shachar Lovett \thanks{Institute for Advanced Study.  E-mail:
    \texttt{slovett@math.ias.edu}.  Supported by NSF grant DMS-0835373.}
    \and Ron Peled \thanks{Tel Aviv University, Israel.
    E-mail: \texttt{peledron@post.tau.ac.il}. Supported by an ISF
    grant and an IRG grant.}}

\maketitle

\begin{abstract} We show the existence of rigid combinatorial objects
which previously were not known to exist. Specifically, for a wide
range of the underlying parameters, we show the existence of non-trivial
orthogonal arrays, $t$-designs, and $t$-wise permutations.  In all cases,
the sizes of the objects are optimal up to polynomial overhead. The proof
of existence is probabilistic.  We show that a randomly chosen such object
has the required properties with positive yet tiny probability. The main
technical ingredient is a special local central limit theorem for suitable
lattice random walks with finitely many steps.
\end{abstract}

\section{Introduction}\label{sec:intro}

We introduce a new framework for establishing the existence of rigid
combinatorial structures, such as orthogonal arrays, $t$-designs and
$t$-wise permutations.  Let $B$ be a finite set and let $V$ be a vector
space of functions from $B$ to the rational numbers $\Q$.  We study when
there is a small subset $T\subset B$ satisfying
\begin{equation} \label{eq:prob_def}
    \frac{1}{|T|}\sum_{t\in T} f(t) = \frac{1}{|B|}\sum_{b\in B}
    f(b)\quad\text{for all $f$ in $V$.}
\end{equation}
In probabilistic terminology, equation \eqref{eq:prob_def} means that if
$t$ is a uniformly random element in $T$ and $b$ is a uniformly random
element in $B$ then
\begin{equation}\label{eq:prob_def_prob}
    \Ex[f(t)] = \Ex[f(b)]\quad\text{for all $f$ in $V$,}
\end{equation}
where $\Ex$ denotes expectation. Of course, \eqref{eq:prob_def} holds
trivially when $T=B$. Our goal is to find conditions on $B$ and $V$
that yield a small subset $T$ that satisfies \eqref{eq:prob_def}, where
in our situations, small will mean polynomial in the dimension of $V$.
(In many natural problems one might encounter a function space $V$ over $\R$
or $\C$ instead.  However, since \eqref{eq:prob_def} is a rational equation,
we can always reduce to the case of rational vector spaces.)

Our main theorem, Theorem~\ref{thm:main}, gives sufficient conditions for
the existence of a small subset $T$ satisfying \eqref{eq:prob_def}. We
apply the theorem to establish results in three interesting cases of
the general framework: orthogonal arrays, $t$-designs, and $t$-wise
permutations.  These are detailed in the next sections. Our methods solve
an open problem, whether there exist non-trivial $t$-wise permutations for
every $t$.  They strengthen Teirlinck's theorem \cite{Teirlinck87}, which
was the first theorem to show the existence of $t$-designs for every $t$.
And they improve existence results for orthogonal arrays, when the size
of the alphabet is divisible by many distinct primes. Moreover, in all
three cases considered, we show the existence of a structure whose size
is optimal up to polynomial overhead.

Our approach to the problem is via probabilistic arguments. In essence, we
prove that a random subset of $B$ satisfies equation \eqref{eq:prob_def}
with positive, albeit tiny, probability. Thus our method is one of the few
known methods for showing existence of rare objects. This class includes
such other methods as the Lov\'asz local lemma \cite{ErdosLovasz73} and
Spencer's ``six deviations suffice'' method \cite{Spencer1985}. However,
our method does not rely on these previous approaches.  Instead, our
technical ingredient is a special version of the (multi-dimensional) local
central limit theorem with only finitely many available steps. Since only
finitely many steps are available, and since we can only gain access to
more steps by increasing the dimension of the random walk, we cannot use
any ``off the shelf'' local central limit theorem, not even one enhanced by
a Berry-Esseen-type estimate of the rate of convergence. Instead, we prove
the local central limit theorem that we need directly using Fourier
analysis. Section~\ref{sec:proof_overview} gives an overview of our
approach.

We also mention that efficient randomized algorithm versions of the
Lov\'asz local lemma \cite{Moser2009, MoserTardos2010} and Spencer's
method \cite{Bansal2010} have recently been found.  Relative to
these new algorithms, the objects that they produce are no longer rare.
Our method is the only one that we know that shows the existence of rare
combinatorial structures, which are still rare relative to any known,
efficient, randomized algorithm.

\subsection{Orthogonal arrays}
\label{sec:OAs}

A subset $T \subset [q]^n$ is an \emph{orthogonal array of alphabet size
$q$, length $n$ and strength $t$} if it yields all strings of length $t$
with equal frequency if restricted to any $t$ coordinates.  In other words,
for any distinct indices $i_1,\ldots,i_t \in [n]$ and any (not necessarily
distinct) values $v_1,\ldots,v_t \in [q]$,
$$
\left|\{x \in T: x_{i_1}=v_1,\ldots,x_{i_t}=v_t\}\right| = q^{-t} |T|.
$$
Equivalently, choosing $x=(x_1,\ldots,x_n) \in T$ uniformly, the distribution
of $x \in [q]^n$ is $t$-wise independent. For an introduction to orthogonal
arrays see~\cite{orthogonalarrays}.

Orthogonal arrays fit into our general framework as follows. We take $B$
to be $[q]^n$ and $V$ to be the space spanned by all functions of the form
\begin{equation}\label{eq:f_I_v_def}
  f_{(I,v)}(x_1,\ldots, x_n) = \begin{cases}
    1& x_i=v_i\text{ for all $i\in I$}\\0&\text{Otherwise}
  \end{cases},
\end{equation}
with $I\subset[n]$ a subset of size $t$ and $v\in[q]^I$.  With this choice,
a subset $T\subset B$ satisfying \eqref{eq:prob_def} is precisely an
orthogonal array of alphabet size $q$, length $n$ and strength $t$.

It is well known that if $T \subset [q]^n$ is $t$-wise independent
then $|T| \ge \left(\frac{cqn}{t}\right)^{t/2}$ for some universal
constant $c>0$ (see, e.g.,~\cite{Rao73}). Matching constructions of
size $|T|\le q^{c t}\left(\frac{n}{t}\right)^{c_q t}$ are known,
however, as these rely on finite field properties the constant $c_q$
generally tends to infinity with the number of prime factors of $q$.
Our technique provides the first upper bound on the size of orthogonal
arrays in which the constant in the exponent is independent of $q$.

\begin{thm}[Existence of orthogonal arrays]\label{thm:OA_first_version}
For all integers $q\ge 2$, $n\ge 1$ and $1\le t\le n$ there exists an
orthogonal array $T$ of alphabet size $q$, length $n$ and strength $t$
satisfying $|T| \le (qn)^{ct}$ for some universal constant $c>0$.
\end{thm}

\subsection{Designs}
\label{sec:designs}

A (simple) \emph{$t$-$(v,k,\lambda)$ design} is a family of distinct
subsets of $[v]$, where each set is of size $k$, such that each $t$
elements belong to exactly $\lambda$ sets. In other words, denoting
by $\choosesq{v}{k}$ the family of all subsets of $[v]$ of size $k$,
a set $T \subset \choosesq{v}{k}$ is a $t$-design if for any
distinct elements $i_1,\ldots,i_t \in [v]$,
\begin{equation}\label{eq:design_def}
\left|\{s \in T: i_1,\ldots,i_t \in s\}\right|
    = \frac{\binom{k}{t}}{\binom{v}{t}} |T|=\lambda.
\end{equation}
For an introduction to combinatorial designs see~\cite{Hand07}.

Our general framework includes $t$-designs as follows. We take $B$ to
be $\choosesq{v}{k}$ and $V$ to be the space spanned by all functions of
the form
\begin{equation}\label{eq:f_a_def}
  f_{a}(b) = \begin{cases}
    1& a\subset b\\0&\text{Otherwise}
  \end{cases},
\end{equation}
with $a\in \choosesq{v}{t}$. With this choice, a subset $T\subset B$
satisfying \eqref{eq:prob_def} is precisely a simple $t-(v,k,\lambda)$
design.

Although $t$-designs have been investigated for many years, the basic
question of existence of a design for a given set of parameters $t,v,k$ and
$\lambda$ remains mostly unanswered unless $t$ is quite small. The case $t=2$
is known as a block design and much more is known about it than for larger
$t$. Explicit constructions of $t$-designs for $t \ge 3$ are known for
various specific constant settings of the parameters (e.g. $5$-$(12,6,1)$
design). The breakthrough result of Teirlinck~\cite{Teirlinck87} was the
first to establish the existence of non-trivial $t$-designs for $t\ge 7$. In
Teirlinck's construction, $k=t+1$ and $v$ satisfies congruences that grow
very quickly as a function of $t$. Other sporadic and infinite examples
have been found since then (see~\cite{Hand07} or \cite{Magliveras09} and
the references within), however, the set of parameters which they cover
is still very sparse. Moreover, it follows from \eqref{eq:design_def}
that any $t-(v,k,\lambda)$ design $T$ has size $|T|=\lambda \binom{v}{t}/
\binom{k}{t} \ge (v/k)^{t}$. Even when existence has been shown, the designs
obtained are often inefficient in the sense that their size is far from
this lower bound.  One of the main results of our work is to establish
the existence of efficient $t$-designs for a wide range of parameters.

\begin{thm}[Existence of $t$-designs]\label{thm:designs_first_version}
For all integers $v\ge 1$, $1\le t\le v$ and $t \le k \le v$ there
exists a $t$-$(v,k,\lambda)$ design whose size is at most $v^{c t}$
for some universal constant $c>0$.
\end{thm}

\subsection{Permutations}
\label{sec:perms}

A family of permutations $T \subset S_n$ is called a \emph{$t$-wise
permutation} if its action on any $t$-tuple of elements is uniform.
In other words, for any distinct elements $i_1,\ldots,i_t \in [n]$ and
distinct elements $j_1,\ldots,j_t \in [n]$,
\begin{equation}\label{eq:perm_def}
    \left| \{\pi \in T: \pi(i_1)=j_1,\ldots,\pi(i_t)=j_t\} \right| =
    \frac{1}{n(n-1)\cdots(n-t+1)} |T|.
\end{equation}

Our general framework includes $t$-wise permutations as follows. We take
$B=S_n$ and $V$ to be the space spanned by all functions of the form
\begin{equation*}
  f_{(i, j)}(b) = \begin{cases}
    1& b(i_1)=j_1,\ldots,b(i_t)=j_t\\0&\text{Otherwise}
  \end{cases},
\end{equation*}
where $i=(i_1,\ldots,i_t)$ and $j=(j_1,\ldots,j_t)$ are $t$-tuples of
distinct elements in $[n]$. With this choice, a subset $T\subset B$
satisfying \eqref{eq:prob_def} is precisely a $t$-wise permutation.

Constructions of families of $t$-wise permutations are known only for
$t=1,2,3$: the group of cyclic shifts $x \mapsto x+a$ modulo $n$ is a
$1$-wise permutation; the group of invertible affine transformations $x
\mapsto ax+b$ over a finite field $\mathbb{F}$ yields a $2$-wise permutation;
and the group of M\"{o}bius transformations $x \mapsto (ax+b)/(cx+d)$
with $ad-bc=1$ over the projective line $\mathbb{F} \cup \{\infty\}$
yields a $3$-wise permutation. For $t \ge 4$ (and $n$ large enough),
however, no $t$-wise permutation is known, other then the full symmetric
group $S_n$ and the alternating group $A_n$ \cite{KaplanNaorReingold05,
AlonLovett11}. In fact, it is known (c.f., e.g., \cite{Cameron95}, Theorem
5.2) that for $n \ge 25$ and $t \ge 4$ there are no other \emph{subgroups}
of $S_n$ which form a $t$-wise permutation. (On other words, there are
no other $t$-transitive subgroups of $S_n$ for $t \ge 4$ and $n \ge 25$.)
One of our main results is to show existence of small $t$-wise permutations
for all $t$.

\begin{thm}[Existence of $t$-wise
permutations]\label{thm:perms_first_version} For all integers $n\ge 1$
and $1\le t\le n$ there exists a $t$-wise permutation $T\subset S_n$
satisfying $|T| \le \exp(t^{c}) n^{c t}$ for some universal constant $c>0$.
\end{thm}
It is clear from the definition~\eqref{eq:perm_def}
above that any $t$-wise permutation $T$ must satisfy $|T|\ge
n(n-1)\cdots(n-t+1)=n^{\Omega(t)}$. Thus, for fixed $t$, the $t$-wise
permutations we exhibit are of optimal size up to polynomial overhead. For
$t$ growing with $n$ these $t$-wise permutations may be larger, but still
no larger than $n^{t^{c}}$ for some universal constant $c>0$.

\subsection{Proof overview}
\label{sec:proof_overview}

The idea of our approach is as follows. Let $T$ be a random multiset of $B$
of some fixed size $N$ chosen by sampling $B$ uniformly and independently
$N$ times (with replacement).  Let $(\phi_a)_{a\in A}$ be a spanning
set of integer-valued functions for $V$ (where $A$ is some finite index
set). Observe that $T$ satisfies \eqref{eq:prob_def} if and only if
\begin{equation}\label{eq:phi_a_expectation}
  \sum_{t\in T} \phi_a(t) = \frac{N}{|B|}\sum_{b\in B}\phi_a(b)
  = \Ex\left[\sum_{t\in T} \phi_a(t)\right] \quad
  \text{for all $a$ in $A$.}
\end{equation}
Thus defining an integer-valued random variable
\begin{equation*}
  X_a:=\sum_{t\in T} \phi_a(t)
\end{equation*}
and $X:=(X_a)_{a\in A} \in \Z^A$ we see that existence of a subset of
size $N$ satisfying \eqref{eq:prob_def} will follow if we can show that
$\Pr[X=\Ex[X]]>0$. To this end we examine more closely the distribution
of $X$. Let $t_1,\ldots, t_N$ be the random elements chosen in forming
$T$. The spanning set $(\phi_a)_{a\in A}$ defines a mapping $\phi:B\to\Z^A$
by the trivial
\begin{equation*}
  \phi(b)_a:=\phi_a(b).
\end{equation*}
Observe that our choice of random model implies that the vectors
$(\phi(t_i))_{i\in [N]}$ are independent and identically distributed. Hence,
\begin{equation}\label{eq:X_as_sum}
    X = \sum_i \phi(t_i)
\end{equation}
may be viewed as the end position of an $N$-step random walk in the lattice
$\Z^{|A|}$. Thus we may hope that if $N$ is sufficiently large, then $X$ has
an approximately (multi-dimensional) Gaussian distribution by the central
limit theorem. If the relevant local central limit theorem holds as well,
then the probability $\Pr[X=x]$ also satisfies a Gaussian approximation. In
particular, since a (non-degenerate) Gaussian always has positive density
at its expectation, we could conclude that $\Pr[X=\Ex[X]]>0$ as desired.

The above description is the essence of our approach. The main obstacle
is, of course, pointed out in the last step. We must control the rate
of convergence of the local central limit theorem well enough that the
convergence error does not outweigh the probability density of the Gaussian
distribution at $\Ex[X]$. Recall that the order of magnitude of such a
density is typically $c^{-|A|}$ for some constant $c>1$, and recall that
$|A|$ is at least the dimension of $V$, which is the main parameter of
our problem.  So we indeed have very small probabilities.  For this reason,
and because we want convergence when $N$ is only polynomial in the dimension
of $V$, we were unable to use any standard local central limit theorem.
Instead, we develop an ad hoc version using direct Fourier analysis.

In our proof of the main theorem, we modify the above description in one
respect.  It is technically more convenient to work with a slightly different
probability model.  Instead of choosing $T$ as above, we set $p:=N/|B|$
and define $T$ by taking each element of $B$ into $T$ independently with
probability $p$.  This has the benefit of guaranteeing that $T$ is a
proper set instead of a multiset. However, it has also the disadvantage
that it does not guarantee that $|T|=N$. To remedy this, we assume that
the space $V$ contains the constant function $h(b) = 1$; or if not, we can
add it to $V$ at the minor cost of increasing the dimension of $V$ by 1.
With this assumption, we note that
\begin{equation*}
  \Ex\bigl[\sum_{t\in T} h(t)\bigr] = \Ex[|T|] = N.
\end{equation*}
Thus \eqref{eq:phi_a_expectation}, or equivalently $X=\Ex[X]$, also
implies that $|T|=N$ as required. Another disadvantage is that in this
new probability model, the vector $X$ is no longer a sum of identically
distributed variables. However, since the summands in \eqref{eq:X_as_sum}
are still independent, we can continue to use Fourier analysis methods in
our proof.

We cannot expect there to always be a small subset $T$ that satisfies
\eqref{eq:prob_def}. For instance, Alon and Vu~\cite{AlonVu97} found a
regular hypergraph with $n$ vertices and $\approx n^{n/2}$ edges, with
no regular sub-hypergraph. Here, the degree of a vertex is the number
of hyperedges incident to it and a regular hypergraph is one in which
the degrees of all vertices are equal. We may describe their example in
our language by letting $B$ be the set of edges of this hypergraph, $A$
be its vertex set, and define $\phi:B\to\{0,1\}^A$ by letting $\phi(b)$ be
the indicator function of the set of vertices incident to $b$. The result
of \cite{AlonVu97} implies that while the vector $\sum_{b\in B} \phi(b)$
is constant, this property is not shared by $\sum_{t\in T} \phi(t)$ for
any non-empty, proper subset $T\subset B$. Thus, we need to impose certain
conditions on $B$ and $V$, or equivalently on the map $\phi$.  We start
by requiring certain divisibility, boundedness and symmetry assumptions.
\begin{description}
\item[Divisibility:] $N$ is such that $\frac{N}{|B|} \sum_{b\in
B}\phi(b)$ is an integer vector. This property is clearly necessary for
\eqref{eq:phi_a_expectation} to hold and is typically a mild restriction
on $N$.

\item[Boundedness:] The entries of $\phi$ must be small. More precisely,
$\max_{a\in A, b\in B}|\phi(b)_a|$ is bounded by a polynomial in $\dim V$,
since our method requires $N$ to be at least some polynomial in this maximum.

\item[Symmetry:] A \emph{symmetry of $\phi$} is a pair consisting of a
permutation $\pi \in S_B$ and a linear transformation $\tau \in \GL(V)$
which satisfies $\phi(\pi(b))=\tau(\phi(b))$ for all $b \in B$. The
set of symmetries $(\pi,\tau)$ of $\phi$ is a subgroup of $S_B \times
\GL(V)$. We require that the projection to $B$ of the group of symmetries
is transitive. In other words, that for any $b_1,b_2 \in B$ there exists
a symmetry $(\pi,\tau)$ of $\phi$ satisfying $\pi(b_1)=b_2$.
\end{description}
It is not hard to verify that the third condition is
intrinsic to the structure of $V$ and does not depend on the
specific choice of spanning set $(\phi_a)$. In our applications it
follows easily from the overall symmetry of the setup.

However, we also have a fourth assumption which is more technical than the
others. First, we require that $(\phi_a)_{a\in A}$ forms a basis of $V$. This
implies that for any $a\in A$, we may express $e_a$, the unit vector with $1$
at its $a$'th coordinate, as a linear combination of the form $\sum_{b\in
B} c_b \phi(b)$. We call any such linear combination an \emph{isolating
combination for $a$}.  We assume that for each $a\in A$, there are many
isolating combinations supported on disjoint subsets of $B$. Moreover, we
require the coefficients of these combinations to have small norm and to
be rational with a small common denominator. This is the most difficult
assumption to verify in our applications. Section~\ref{sec:framework}
gives more details about all of these assumptions.

Our main theorem shows that these four conditions yield the existence of
a small solution of \eqref{eq:prob_def}.
\begin{thm*}[Main theorem - informal statement] Let $B$ be a finite set and
let $V$ be a vector space of functions from $B$ to $\Q$ which contains the
constant functions. If there exists a basis $(\phi_a)_{a \in A}$ of $V$,
consisting of integer-valued functions, which satisfies the boundedness,
symmetry and isolation conditions above. Then there is a small subset
$T \subset B$ such that
$$
\frac{1}{|T|} \sum_{t \in T} f(t) = \frac{1}{|T|} \sum_{b \in B}
f(b)
$$
for all $f$ in $V$.
\end{thm*}
We note that the size $N=|T|$ of the subset obtained must satisfy the
divisibility condition above. The existence theorems for orthogonal arrays,
$t$-designs and $t$-wise permutations follow by showing that for the choice
of $B$ and $V$ detailed in Sections~\ref{sec:OAs} through \ref{sec:perms}
there exists a choice of basis $\{\phi_a\}$ and small $N$ for which all
four conditions above hold.

\subsection{Related work}

In the probabilistic formulation \eqref{eq:prob_def_prob} of our problem we
seek a small subset $T \subset B$ such that the uniform distribution over
$T$ simulates the uniform distribution over $B$ with regards to certain
tests. There are two ways to relax the problem to make its solution easier,
and raise new questions regarding explicit solutions.

One relaxation is to allow a set $T$ with a non-uniform distribution
$\mu$. For many practical applications of $t$-designs and $t$-wise
permutations in statistics and computer science, but not quite every
application, this relaxation is as good as the uniform question.
The existence of a solution with small support is guaranteed by
Carath\'eodory's theorem, using the fact that the constraints
on $\mu$ are all linear equalities and inequalities. Moreover,
such a solution can be found efficiently, as was shown by Karp
and Papadimitriou~\cite{KarpPapadimitriou82} and in more general
settings by Koller and Megiddo~\cite{KollerMegiddo94}. Alon and
Lovett~\cite{AlonLovett11} give a strongly explicit analog of this in the case
of $t$-wise permutations and more generally in the case of group actions.

A different relaxation is to require the uniform distribution on $T$ to
only approximately satisfy equation \eqref{eq:prob_def_prob}.  Then it is
trivial that a sufficiently large random subset $T \subset B$ satisfies
the requirement with high probability, and the question is to find an
explicit solution. For instance, we can relax the problem of $t$-wise
permutations to \emph{almost} $t$-wise permutations.  For this variant
an optimal solution (up to polynomial factors) was achieved by Kaplan,
Naor and Reingold~\cite{KaplanNaorReingold05}, who gave a construction of
such an almost $t$-wise permutation of size $n^{O(t)}$.  Alternatively,
one can start with the constant size expanding set of $S_n$ given by
Kassabov~\cite{Kassabov07} and take a random walk on it of length $O(t
\log{n})$.

\subsection{Paper organization}

We give a precise description of the general framework and our main theorem
in Section~\ref{sec:framework}. We apply it to show the existence of
orthogonal arrays and $t$-designs in Section~\ref{sec:applications}. The
case of $t$-wise permutations requires a detour to the representation
theory of the symmetric group, and we defer it to the full version of
this paper. The proof of our main theorem is given in Section~\ref{sec:proof_main_theorem}. 
We summarize and give some open
problems in Section~\ref{sec:summary}.

\section{Main Theorem}
\label{sec:framework}

Let $B$ be a finite set and let $V$ be a vector space of functions from
$B$ to $\Q$. We ask for conditions for the existence of a small set $T
\subset B$ for which \eqref{eq:prob_def} holds.  Our theorem uses the
following notation.

For a basis $(\phi_a)_{a\in A}$ (where $A$ is some finite index set) of $V$
we define $\phi:B \to \Z^A$ by $\phi(b)_a = \phi_a(b)$. This definition
is extended linearly to $\phi:\Z^B\to\Z^A$ by setting $\phi(\gamma) =
\sum_{b \in B} \gamma_b \phi(b)$. In the same manner, a set $T\subset B$
is identified with its indicator vector so that $\phi(T)=\sum_{t \in T}
\phi(t)$. Finally, we recall from Section~\ref{sec:proof_overview} that a
symmetry of $\phi$ is a pair $\pi \in S_B$ and $\tau \in \GL(V)$ such that
$\phi(\pi(b)) = \tau(\phi(b))$ for all $b$ in $B$. We now state formally
our main theorem.

\begin{thm}[Main Theorem]\label{thm:main} Let $B$ be a finite set and $V$
be a vector space of functions from $B$ to $\Q$ which contains the constant
functions. Suppose that there exist integers $m,c_0\ge 1$, real numbers
$c_1,c_2,c_3 > 0$ and a basis $(\phi_a)_{a \in A}$ of $V$ consisting of
integer-valued functions such that:
\begin{description}
\item[Divisibility:] $\frac{c_0}{|B|} \phi(B)$ is an integer vector.
\item[Boundedness:] $\|\phi(b)\|_{2} \le c_1$ for all $b \in B$.
\item[Symmetry:] For each $b_1,b_2 \in B$ there exists a symmetry
$(\pi,\tau)$ of $\phi$ such that $\pi(b_1)=b_2$.
\item[Isolation:] For any $a \in A$ there exist vectors
$\gamma_1,\ldots,\gamma_r \in \Z^B$ for $r \ge |B|/c_2$ such that
\begin{itemize}
\item $\phi(\gamma_i) = m \cdot e_a$ for all $i \in [r]$.
\item The vectors $\gamma_1,\ldots,\gamma_r$ have disjoint supports,
where the support of a vector $\gamma \in \Z^B$ is the set of coordinates
on which it is nonzero.
\item $\|\gamma_i\|_2 \le c_3$ for all $i \in [r]$.
\end{itemize}
\end{description}
Then there exists a subset $T \subset B$ with $|T| \le
\poly(|A|,m,c_0,c_1,c_2,c_3)$ such that
$$
\frac{1}{|T|}\sum_{t\in T} f(t) = \frac{1}{|B|}\sum_{b\in B}
f(b)\quad\text{for all $f$ in $V$.}
$$
\end{thm}
We prove Theorem~\ref{thm:main} in Section~\ref{sec:proof_main_theorem}.
A careful examination of the proof shows that we can choose $|T|=N$ for
any $N \ge 1$ which satisfies the following constraints:
\begin{itemize}
\item $c_0 m$ divides $N$;
\item $N \ge \Omega(1) \cdot \max(m^3, |A|^2 m^2 \log^2(|A| m c_0 c_1 c_2
c_3), |A|^6 c_1^6 c_2^3 c_3^6 \log^3(|A| m c_0 c_1 c_2 c_3))$;
\item $N \le O(\sqrt{|B|})$.
\end{itemize}
Of course, if the parameters are so large so that the second and third
conditions contradict each other, then our theorem remains trivially true
by taking $T=B$.

\section{Applications}
\label{sec:applications}

In this section we apply our main theorem, Theorem~\ref{thm:main},
to prove the existence results for orthogonal arrays
and $t$-designs, Theorems~\ref{thm:OA_first_version} and
\ref{thm:designs_first_version}. The existence result for $t$-wise
permutations, Theorem~\ref{thm:perms_first_version}, is more complicated
because it requires a discussion of the representation theory of the
symmetric group.  We defer it to the full version of this paper.

\subsection{Orthogonal arrays}
\label{sec:orthogonal_arrays}

We use the choice of $B$ and $V$ described in Section~\ref{sec:OAs} and
recall the definition~\eqref{eq:f_I_v_def} of the functions $f_{(I,v)}$ of
that section. We note that for every subset $I$ we have $\sum_{v\in[q]^I}
f_{(I,v)}\equiv 1$. Thus $V$ contains the constant functions as
Theorem~\ref{thm:main} requires. We start by choosing a convenient
basis for $V$ of integer-valued functions.  Recall that the alphabet is
$[q]=\{1,\ldots,q\}$ and let $[q-1]=\{1,\ldots,q-1\}$ be all symbols other
than $q$. Extend the definition~\eqref{eq:f_I_v_def} of $f_{(I,v)}$ to
apply to all subsets $I$ with $|I| \le t$ and $v \in [q]^I$. Here, we mean
that $f_{(\emptyset,\emptyset)}$ is the constant function $1$. Finally, let
$$
A:=\{(I,v): |I| \le t, v \in [q-1]^{|I|}\}
$$
and for $a=(I,v) \in A$ set $\phi_a := f_{(I,v)}$.

\begin{claim} \label{cl:OA_A_spans_V} The span of the functions
$\{\phi_a\}_{a \in A}$ is $V$.
\end{claim}

\begin{proof}
Clearly $\phi_a \in V$ for all $a \in A$. To see that $\{\phi_a\}_{a \in A}$
spans $V$, we will show that any $f_{(I,v)}$ with $|I| \le t$ and $v \in
[q]^I$ is spanned by $\{\phi_a\}_{a \in A}$. We do this by induction on the
number of elements in $v$ which are equal to $q$. First, if $v \in [q-1]^I$
then $f_{(I,v)}=\phi_{(I,v)}$.  Otherwise, let $I=\{i_1,\ldots,i_r\}$
with $r \le t$, $v \in [q]^I$ and assume WLOG that $v_{i_1}=q$. Then
$$
f_{(I,v)} = f_{(\{i_2,\ldots,i_r\},(v_{i_2},\ldots,v_{i_r}))} -
\sum_{m=1}^{q-1} f_{(I,(m,v_{i_2},\ldots,v_{i_r}))}
$$
and by induction, the right hand side belongs to the linear span of
$\{\phi_a\}_{a \in A}$.
\end{proof}

Recall that $\phi:B \to \Z^A$ is defined as $\phi(b)_a = \phi_a(b)$.
We now choose integers $m,c_0\ge 1$ and real numbers $c_1,c_2,c_3 >0$ such
that the conditions of divisibility, boundedness, symmetry and isolation
required by Theorem~\ref{thm:main} are satisfied.  First, let $a=(I,v)
\in A$. Note that $\frac{1}{|B|}\phi(B)_a=q^{-|I|}$. Thus we set $c_0=q^t$
so that $\frac{c_0}{|B|} \phi(B)$ is an integer vector. Second, we clearly
have for any $b \in B$ that $\|\phi(b)\|_2^2=\sum_{i=0}^t \binom{n}{i}
\le (n+1)^t$. Hence we set $c_1=(n+1)^{t/2}$.

Third, to witness the symmetry condition, fix $x \in [q]^n$ and consider the
permutation $\pi \in S_B$ given by $\pi(b)=b+x \pmod{q}$. We need to show
that there exists a linear map $\tau$ acting on $V$ such that $\phi(\pi(b))
= \tau (\phi(b))$ for all $b\in B$. This holds since for $a=(I,v) \in A$
we have
$$
\phi(\pi(b))_a = f_{I,v}(b+x \textrm{ (mod }{q}))=f_{I,v-x\textrm{
    (mod }q)}(b)
$$
and $f_{I,v-x\textrm{ (mod }q)} \in V$ is in the linear span of
$\{\phi_a\}_{a \in A}$ by Claim~\ref{cl:OA_A_spans_V}.

The fourth condition we need to verify is the existence of many disjoint
isolation vectors for each $a \in A$. Note that this condition also implies
that $\{\phi_a\}_{a \in A}$ is a basis for $V$. This is established in
the following lemma.

\begin{lemma}\label{lem:OA_gamma_vecs}
Let $a \in A$. There exist disjoint vectors $\gamma_1,\ldots,\gamma_r \in
\Z^B$ with $r \ge |B|/(q^t n^{2t})$ and $\|\gamma_i\|_2 \le 2^{3t/2} n^t$
such that $\phi(\gamma_i)=e_a$.
\end{lemma}

We prove Lemma~\ref{lem:OA_gamma_vecs} in two steps. First we fix some notations.  Let $K \subset [n]$ be of size $|K| \le
t$, and let $K^c=[n] \setminus K$. For $x \in [q]^n$ let $x|_K \in [q]^{K}$
be the restriction of $x$ to the coordinates of $K$.  Abusing notation, we
also think of $x|_K \in [q]^n$ by setting coordinates outside $K$ to zero.
Note that in this notation, $f_{I,v}(x) = \one{x|_I=v}$. We define the
vector $\delta_{x,K} \in \Z^B$ as
$$
\delta_{x,K} := \sum_{J \subseteq K} (-1)^{|K|-|J|} e_{x|_{J \cup K^c}},
$$
where we recall that for $b \in B$, $e_b \in \{0,1\}^{B}$ is
the corresponding unit vector. Note that if $K=\emptyset$ then
$\delta_{x,\emptyset}=e_x$.

\begin{claim}\label{cl:OA_delta_vecs}
Let $a=(I,v) \in A$. Then
$$
\phi(\delta_{x,K})_{a} = \Bigg\{
\begin{array}{ll}
0&\textrm{if } K \not \subseteq I\\
0&\textrm{if } a|_K \ne x|_K\\
1&\textrm{if } a|_K=x|_K\\
\end{array}
$$
\end{claim}

\begin{proof}
We compute the value of $\phi(\delta_{x,K})$ in coordinate $a=(I,v) \in
A$. We have
$$
\phi(\delta_{x,K})_{a} = \sum_{J \subseteq K} (-1)^{|K|-|J|} \one{(x|_{J
    \cup K^c})|_{I}=v}.
$$
Suppose first that $K \not \subseteq I$. Then there exists $j \in K
\setminus I$. Flipping the $j$-th element in $J$ doesn't change the
expression $\one{(x|_{J \cup K^c})|_I=v}$ and hence the alternating sign
sum cancels. We thus assume from now on that $K \subseteq I$. We thus have
\begin{align*}
\one{(x|_{J \cup K^c})|_I=v} = \one{x|_J=v|_K \textrm{ and } x|_{K^c
    \cap I}=v|_{K^c}}.
\end{align*}
This expression evaluates to $1$ only if $J=K$ and $x|_I=v$.
\end{proof}

We next prove Lemma~\ref{lem:OA_gamma_vecs}, showing that we can build
many disjoint isolation vectors for any $a \in A$. The proof uses the
vectors $\delta_{x,K}$ we just analyzed.

\begin{proof}[Proof of Lemma~\ref{lem:OA_gamma_vecs}]
Fix $a=(I,v)$. Let $x \in [q]^n$ be such that $x|_I=v$. We will construct
a vector $\gamma_{x,I}$ such that $\phi(\gamma_{x,I})=e_a$. We will do so
by backward induction on $|I| \le t$. If $|I|=t$ we take
$$
\gamma_{x,I} := \delta_{x,I},
$$
and if $|I|<t$ we construct recursively
$$
\gamma_{x,I} := \delta_{x,I} - \sum_{K \supsetneq I, |K| \le t, x_{K} \in
    [q-1]^{K}} \gamma_{x,K}.
$$
It is easy to verify using Claim~\ref{cl:OA_delta_vecs} that
indeed $\phi(\gamma_{x,I})=e_a$ as claimed. We further claim that
$\|\gamma_{x,I}\|_2 \le 2^{t/2} (2n)^{t-|I|}$. This clearly holds if
$|I|=t$. If $|I|<t$ we bound by induction
\begin{align*}
\|\gamma_{x,I}\|_2 &\le \|\delta_{x,I}\|_2 +
    \sum_{k=|I|+1}^t \sum_{K \supset I, |K|=k} \|\gamma_{x,K}\|_2 \\
&\le 2^{t/2} \left(1 + \sum_{k=|I|+1}^t
    \binom{n-|I|}{k-|I|} (2n)^{t-k} \right) \\
&\le 2^{t/2} \left(1 + \sum_{k=|I|+1}^t
    n^{k-|T|} (2n)^{t-k} \right) \\
&\le 2^{t/2} n^{t-|I|} \left(1 +
    \sum_{k=|I|+1}^t (2)^{t-k} \right)=2^{t/2}(2n)^{t-|I|}.
\end{align*}
To conclude, we need to show that by choosing different values for $x$
such that $x|_I=v$ we can achieve many disjoint vectors which isolate
$a$. The key observation is that $\gamma_{x,I}$ is supported on elements
$b \in B$ whose hamming distance from $x$ is at most $t$. Thus, if we
choose $x_1,\ldots,x_r \in [q]^n$ such that $(x_i)|_I = v$ and such that
the hamming distance between each pair $x_i,x_j$ is at least $2t+1$, we
get that $\gamma_{x_1,I},\ldots,\gamma_{x_r,I}$ have disjoint supports. We
can achieve $r \ge q^{n-t} / n^{2t}$ by a simple greedy process: choose
$x_1,\ldots,x_r$ iteratively; after choosing $x_i$ delete all elements in
$[q]^n$ whose hamming distance from $x_i$ is at most $2t$. Since the number
of these elements is bounded by $\sum_{i=1}^{2t} \binom{n}{i} \le n^{2t}$
the claim follows.
\end{proof}

We now have all the conditions to apply Theorem~\ref{thm:main}. We have $|A|=\sum_{i=0}^t
\binom{n}{i}(q-1)^i \le (q(n+1))^t, c_0=q^t, c_1=(n+1)^{t/2}, c_2=q^t
n^{2t}, c_3=2^{3t/2} n^t$ and $m=1$. Hence we obtain that there exists
an orthogonal array $T \subset [q]^n$ of strength $t$ and size $|T|\le
(qn)^c$ for some universal constant $c>0$.

\subsection{Designs}
\label{sec:designsproof}

In this section, we prove Theorem~\ref{thm:designs_first_version}. It
suffices to prove the theorem for $k>2t$, since if $k\le 2t$ then
the complete design (the design containing all subsets of size $k$)
establishes the theorem. We use the choice of $B$ and $V$ described in
Section~\ref{sec:designs} and recall the definition~\eqref{eq:f_a_def}
of the functions $f_{a}$ of that section. We set $A=\choosesq{v}{t}$ and
note that $\sum_{a\in A} f_{a}\equiv\choosesq{k}{t}$ and thus $V$ contains
the constant functions as Theorem~\ref{thm:main} requires. As a convenient
basis for $V$ of integer-valued functions, we take $\{\phi_a\}_{a\in A}$
with $\phi_a=f_a$. By definition, $\{\phi_a\}_{a\in A}$ spans $V$ and
the fact that $\{\phi_a\}_{a\in A}$ is a basis for $V$ will be implied by
showing the isolation condition of Theorem~\ref{thm:main}.

We choose integers $m,c_0\ge 1$ and real numbers $c_1,c_2,c_3
> 0$ to satisfy the conditions of divisibility, boundedness,
symmetry and isolation in Theorem~\ref{thm:main}.  First,
$\frac{1}{|B|}\phi(B)=\binom{k}{t}/\binom{v}{t} \cdot (1,\ldots,1)$
and hence we set $c_0 = \binom{v}{t}$ so that $\frac{c_0}{|B|} \phi(B)$
is an integer vector. Second, $\|\phi(b)\|_2^2 \le |A| \le v^t$. Hence we
set $c_1=v^{t/2}$.  Third, the symmetry condition also follows simply:
let $\sigma \in S_v$ be a permutation on $[v]$. It acts naturally on
$B$ and $A$ (by permuting subsets of $[v]$) and gives two permutations
$\pi \in S_B$ and $\tilde{\pi} \in S_A$ that satisfy $\phi(\pi(b))_a =
\phi(b)_{\tilde{\pi}^{-1}(a)}$. The linear transformation $\tau \in \GL(V)$
then corresponds to the permutation $\tilde{\pi}^{-1}$.

Finally, we need to show that for each $a \in A$ there exist many disjoint
vectors which isolate it. This is accomplished in the following lemma.

\begin{lemma}\label{lem:gamma_vectors_designs}
Assume $k>2t$. For any $a\in A$ there exist vectors
$\gamma_1,\ldots,\gamma_r \in \Z^B$ with $r \ge |B|/(vk)^{2t}$
such that $\phi(\gamma_i)=\frac{k!}{(k-t)!} \cdot e_a$. Moreover,
$\gamma_1,\ldots,\gamma_r$ have disjoint supports and $\|\gamma_i\|_2 \le
(2k)^{3t/2}$ for $i \in [r]$.
\end{lemma}

We will need the following technical claim for the proof of
Lemma~\ref{lem:gamma_vectors_designs}. In the following we consider
binomial coefficients $\binom{n}{m}=0$ whenever $n<m$.

\begin{claim}\label{cl:binomial_sum_designs}
Let $a>b \ge 0$ and $c \ge 0$. Then
$$
\sum_{i=0}^a (-1)^{i} \binom{a}{i} \binom{c+i}{b}=0.
$$
\end{claim}

\begin{proof}
Let $f(a,b,c) = \sum_{i=0}^a (-1)^i \binom{a}{i} \binom{c+i}{b}$. If
$b,c>0$ we have $\binom{c+i}{b}=\binom{c-1+i}{b}+\binom{c-1+i}{b-1}$ and
hence $f(a,b,c) = f(a,b,c-1)+f(a,b-1,c-1)$.  So, it is enough to verify the
claim whenever $b=0$ or $c=0$. If $b=0$ then $f(a,0,c)=\sum_{i=0}^a (-1)^i
\binom{a}{i}=0$ since $a \ge 1$. If $c=0$ then $f(a,b,0)=\sum_{i=b}^a
(-1)^i \binom{a}{i} \binom{i}{b} = \binom{a}{b} \sum_{i=b}^a (-1)^i
\binom{a-b}{i-b}=0$.
\end{proof}

\begin{proof}[Proof of Lemma~\ref{lem:gamma_vectors_designs}]
Let $a \in A=\choosesq{v}{t}$ be a coordinate we wish to isolate. Let $x \in
\choosesq{v}{k}$ be a set disjoint from $a$ and let $0 \le j \le t$. Define
$\delta_{x,a,j} \in \Z^B$ to be the indicator vector for all subsets $b \in
B=\choosesq{v}{k}$ such that $b \subset a \cup x$ and $|a \cap b|=j$, that is
$$
\delta_{x,a,j} := \sum_{b \subset a \cup x, |b|=k, |a \cap b|=j} e_b.
$$
We define vectors $\gamma_{a,x} \in \Z^B$ as
$$
\gamma_{x,a} := \sum_{j=0}^t (-1)^{t-j} \frac{j! (k-j-1)!}{(k-t-1)!}
\delta_{x,a,j}.
$$
We will shortly show that
$$
\phi(\gamma_{x,a}) = \frac{k!}{(k-t)!} e_a.
$$
First we bound the norm of $\gamma_{x,a}$ and show the existence of
many disjoint vectors. It is easy to check that $\|\gamma_{x,a}\|_2 \le
(2k)^{3t/2}$. Also, the vector $\gamma_{x,a}$ is supported on coordinates $y
\in B$ such that $|y \cap x| \ge k-t$. Thus, if we choose $x_1,\ldots,x_r
\in B$ such that $|x_i \cap x_j| \le k-2t-1$ we get that the vectors
$\gamma_{x_1,a},\ldots,\gamma_{x_r,a}$ have disjoint support. We can choose
$r \ge |B|/(vk)^{2t}$ by a simple greedy argument: choose $x_1,\ldots,x_r$
iteratively, where in each step after choosing $x_i$ we remove all subsets
$y \in B$ whose intersection with $x_i$ is at least $k-2t$. The number of
subsets eliminated in each step is at most $(vk)^{2t}$ hence we will get
$r \ge |B|/(vk)^{2t}$.

To conclude the proof, we need to compute $\phi(\gamma_{x,a})$. Let
$a' \in A$. Clearly if $a' \not \subseteq a \cup x$ then
$\phi(\gamma_{x,a})_{a'}=0$. We thus assume that $a' \subset a \cup
x$. Let $\ell=|a \cap a'|$ where $0 \le \ell \le t$. We have that
$\phi(\delta_{x,a,j})_{a'}=0$ if $j<\ell$, and that
\begin{align*}
\phi(\delta_{x,a,j})_{a'} &= |\{b \in B: a' \subset b \subset a \cup x,
    |a \cap b|=j\}|\\
&= \binom{t-\ell}{t-j}\binom{k-t+\ell}{j}.
\end{align*}
Hence we have that
\begin{equation}\label{eq:gamma_x_a_atag}
\phi(\gamma_{x,a})_{a'} = \frac{(k-1)!}{(k-t-1)!} \sum_{j=\ell}^t (-1)^{t-j}
    \frac{\binom{t-\ell}{t-j}\binom{k-t+\ell}{j}}{\binom{k-1}{j}}
\end{equation}
If $a'=a$ then $\phi(\gamma_{x,a})_a = k!/(k-t)!$ as claimed. To conclude
we need to prove that if $a' \ne a$ then $\phi(\gamma_{x,a})_{a'} = 0$. We
have $\ell=|a \cap a'|<t$ and let $s=t-\ell>0$. Thus
\begin{align*}
\phi(\gamma_{x,a})_{a'} &= (-1)^t \frac{(k-1)!}{(k-t-1)!}
    \sum_{j=\ell}^t (-1)^{j} \frac{\binom{t-\ell}{t-j}
    \binom{k-t+\ell}{j}}{\binom{k-1}{j}}\\
&= (-1)^{s} \frac{(k-1)!}{(k-t-1)!} \sum_{j=0}^{s} (-1)^j
    \frac{\binom{s}{j}\binom{k-s}{j+\ell}}{\binom{k-1}{j+\ell}}\\
&= (-1)^{s} \frac{(k-1)!}{(k-t-1)!\binom{k-1}{s-1}}
    \sum_{j=0}^{s} (-1)^j \binom{s}{j}\binom{k-\ell-1-j}{s-1}\\
&= (-1)^{s} \frac{(k-1)!}{(k-t-1)!\binom{k-1}{s-1}}
    \sum_{j=0}^{s} (-1)^j \binom{s}{j}\binom{k-\ell-1-s+j}{s-1}.
\end{align*}
We now apply Claim~\ref{cl:binomial_sum_designs} with
$a=s,b=s-1,c=k-\ell-1-s$ and conclude that $\phi(\gamma_{x,a})_{a'}=0$.
\end{proof}

We are now ready to
apply Theorem~\ref{thm:main}. We have $|A|=\binom{v}{t},c_0=\binom{v}{t},
c_1=v^{t/2}, c_2=(vk)^{2t}, c_3=(2k)^{3t/2}$ and $m=k!/(k-t)!$. Thus the
theorem implies the existence of a $t-(v,k,\lambda)$ design $T \subset B$
with $|T|\le v^{ct}$ for some universal constant $c>0$.

\section{Proof of Main Theorem}
\label{sec:proof_main_theorem}

We prove Theorem~\ref{thm:main} in this section. We recall the settings:
$B$ is a finite set and $V$ is a vector space of functions from $B$
to $\Q$. We assume the space $V$ is spanned by integer valued functions
$\{\phi_a:B \to \Z\}_{a \in A}$, where $A$ is a finite index set. We also
assume that the constant functions belong to $V$.

The proof strategy is conceptually simple: choose $T$ randomly and show
that this choice is successful with positive probability. Let $N$ be
the target size of $T$, to be chosen later. Let each $b \in B$ be chosen
to be in $T$ independently with probability $p:=N/|B|$. Identifying $T$
with its indicator vector in $\{0,1\}^B$, we have that $T_b \in \{0,1\}$
with $\Pr[T_b=1]=p$. Define $X=\phi(T) \in \Z^A$ and note that $\Ex[X]=p
\cdot \phi(B)$. In order to prove Theorem~\ref{thm:main} we need to show that
\begin{equation}\label{eq:pr_X_EX}
\Pr[X = \Ex[X]]>0.
\end{equation}
We make two notes: first, since we assume that constant functions belong
to $V$ we have that if $X=\Ex[X]$ then in particular $|X|=p|B|=N$. Second,
in order for~\eqref{eq:pr_X_EX} to hold we must have that $\Ex[X]$ is an
integer vector. Thus, we must choose $N$ to be divisible by $c_0$.

The difficulty with establishing~\eqref{eq:pr_X_EX} comes from the fact
that we require $A$ different events to occur simultaneously: for all $a
\in A$ we require that $X_a = \Ex[X_a]$. To better explain the challenge,
consider momentarily for simplicity the case where $\phi(b) \in \{0,1\}^A$
for all $b \in B$ and that for each $a \in A$, $\Pr_{b \in B}[\phi(b)_a=1]=q$
(that is, all columns of $\phi$ have $qB$ ones). Then each individual $X_a$
is binomially distributed, $X_a \sim \mathrm{Bin}(|B|, pq)$, and it is
not hard to see that
$$
\Pr[X_a = \Ex[X_a]] \approx \frac{1}{\sqrt{q N}}.
$$
However, we need the events $X_a=\Ex[X_a]$ to occur simultaneously for all $a
\in A$. The problem arises because these events are dependent, and general
techniques for handling such dependencies (for example, the Lov\'{a}sz
local lemma) only work when each event depends only on a few other events
(which is not the case here) and where each event holds with sufficiently
high probability (which is also not the case here). What we show is that,
under the conditions of Theorem~\ref{thm:main}, if we choose $N$ large
enough (but only polynomially large in $|A|,m,c_0,c_1,c_2,c_3$) then all
the events $X_a=\Ex[X_a]$ become essentially independent, and we show that
$$
\Pr[X = \Ex[X]] \approx \prod_{a \in A}
    \Pr[X_a = \Ex[X_a]] \approx (\frac{1}{\sqrt{qN}})^{|A|}.
$$
The actual expression we get is somewhat more complicated as it also
involves pairwise correlations between the different events $X_a$, but
conceptually it is of a similar flavor.

Our main technique to study the distribution of the random variable $X
\in \Z^A$ is Fourier analysis. We recall some basic facts about Fourier
analysis on $\Z^A$.

\begin{fact}[Fourier analysis on $\Z^A$]
Let $X \in \Z^A$ be a random variable. The Fourier coefficients of $X$ live
in the $A$-dimensional torus. Let $\T=[-1/2,1/2)$ denote the torus. The
Fourier coefficients $\hat{X}(\theta)$ for $\theta \in \T^A$ are given by
$$
\hat{X}(\theta) = \Ex_X[e^{2 \pi i \ip{X,\theta}}],
$$
where $\ip{X,\theta}=\sum_{a \in A} X_a \theta_a$. The probability that
$X=\lambda$ for $\lambda \in \Z^A$ is given by the Fourier inversion formula
$$
\Pr[X=\lambda] = \int_{\theta \in \T^A} \hat{X}(\theta)
    e^{-2 \pi i \ip{\lambda,\theta}} d \theta.
$$
\end{fact}

Recall that our goal is to understand the probability that
$X=\Ex[X]$. Applying the Fourier inversion formula for $\lambda=\Ex[X]$ gives
\begin{equation}\label{eq:fourier_inversion}
\Pr[X=\Ex[X]] = \int_{\theta \in \T^A} \hat{X}(\theta)
    e^{-2 \pi i \ip{\Ex[X],\theta}} d \theta.
\end{equation}
Thus, our goal from now on is to understand the Fourier coefficients of $X$.
We first give an explicit formula for the Fourier coefficients.

\begin{claim}\label{cl:structure_fourier_X}
We have
$$
\hat{X}(\theta) = \prod_{b \in B}
    (1-p+p e^{2 \pi i \cdot \ip{\phi(b),\theta}}).
$$
\end{claim}

\begin{proof}
By definition $X=\phi(T)=\sum_{b \in B} T_b \phi(b)$, where $T_b \in \{0,1\}$
are independent with $\Pr[T_b=1]=p$. Thus
\begin{align*}
\hat{X}(\theta) &= \Ex_X[e^{2 \pi i \ip{X,\theta}}] = \Ex_{\{T_b: b \in B\}}
    [e^{2 \pi i \sum_{b \in B} T_b \ip{\phi(b),\theta}}]\\
&= \prod_{b \in B} \Ex_{T_b}[e^{2 \pi i \; T_b \ip{\phi(b),\theta}}]
= \prod_{b \in B} (1-p+pe^{2 \pi i \ip{\phi(b),\theta}}).
\end{align*}
\end{proof}

Clearly all Fourier coefficients of $X$ have absolute value at most
$1$. The first step is to understand the maximal Fourier coefficients of
$X$, that is $\theta$ for which $|\hat{X}(\theta)|=1$.

\begin{claim}
\label{cl:properties_L}
Let $L:=\{\theta \in \T^A: \hat{X}(\theta)=1\}$. Then
\begin{itemize}
\item If $\theta \notin L$ then $|\hat{X}(\theta)|<1$.
\item If $\theta \in L, \theta' \in \T^A$ then
$\hat{X}(\theta+\theta')=\hat{X}(\theta')$. In particular, $L$ is a subgroup
of $\T^A$.
\end{itemize}
\end{claim}

\begin{proof} Both claims follow immediately from the observation that
$\theta \in L$ iff $\ip{\phi(b),\theta} \in \Z$ for all $b \in B$.
\end{proof}

In fact, the isolation conditions in Theorem~\ref{thm:main} imply that $L$
is a discrete subgroup of $\T^A$ (i.e. a lattice). Let $M:=(1/m \cdot
\Z)^A$ be the lattice in $\T^A$ of all elements whose coordinates are
integer multiplies of $1/m$. We show that $L$ is a sublattice of $M$.

\begin{claim}\label{cl:L_in_M}
$L \subseteq M$.
\end{claim}

\begin{proof} Let $\theta \in L$. We need to show that $m \theta_a \in \Z$
for all $a \in A$. By the isolation condition of Theorem~\ref{thm:main},
there exists $\gamma \in \Z^B$ such that $\phi(\gamma)=m e_a$. Since $\theta
\in L$ we have that $\ip{\phi(b),\theta} \in \Z$ for all $b \in B$. Hence
also $\ip{\phi(\gamma),\theta} \in \Z$, i.e. $m \theta_a \in \Z$ as claimed.
\end{proof}

The first step we take is to approximate the Fourier coefficients of $X$
near the lattice $L$. This will assume very little about $\phi$, essentially
only boundedness. The second (and more complex) step will be to show that
all other Fourier coefficients are negligible, and in fact the contribution
to $\eqref{eq:fourier_inversion}$ all come from Fourier coefficients near
$L$. The second part will heavily utilize the symmetry of the map $\phi$
and the existence of many disjoint isolation vectors. Theorem~\ref{thm:main}
then follows by a careful setting of parameters and a routine calculation.

Formally, we will use $\ell_2$ distance on $\T^A$. For $x
\in \T$ define its absolute value $|x|=|x \pmod{1}|$ to be the minimal
absolute value of $x$ modulo $1$ (that is, we take $x \pmod{1} \in
[-1/2,1/2]$). Define the distance between $\theta',\theta'' \in \T^A$ by
$$
d(\theta',\theta'') := \sqrt{\sum_{a \in A} |\theta'_a-\theta''_a|^2}.
$$
The distance between $\theta \in \T^A$ and $L \subset \T^A$ is given by
$$
d(\theta,L) := \min_{\alpha \in L} d(\theta,\alpha).
$$

The following three lemmas are the main technical ingredients of the
proof. The first lemma gives a good approximation for the Fourier
coefficients of $X$ near zero (and by Claim~\ref{cl:properties_L},
near any point in $L$).

\begin{lemma}[Estimating Fourier coefficients near zero]
\label{lem:estimate_fourier_near_zero}
Assume the conditions of Theorem~\ref{thm:main} and fix $\eps \le O(1/(c_1
N^{1/3}))$. Let $\theta \in \T^A$ be such that $\|\theta\|_2 \le \eps$. Then
$$
\hat{X}(\theta) = e^{2 \pi i \ip{\Ex[X],\theta}}
    e^{-4 \pi^2 p \cdot \theta^T R \theta}(1+\delta)
$$
where $R$ is the $A \times A$ pairwise-correlation matrix of $\phi$ given
by $R_{a',a''}=\sum_{b \in B} \phi(b)_{a'} \phi(b)_{a''}$, and where
$|\delta|=O(N^2/|B|+N c_1^3 \eps^3)$.
\end{lemma}

The second lemma bounds the Fourier coefficients of $X$ far from the
lattice $M$.

\begin{lemma}[Bounding Fourier coefficients far from M]
\label{lem:bound_fourier_far_M}
Assume the conditions of Theorem~\ref{thm:main}. Let $\theta \in \T^A$
be such that $d(\theta,M) \ge \eps$. Then
$$
|\hat{X}(\theta)| \le
    \exp\left(-N \eps^2 \cdot \frac{m^2}{|A| c_2 c_3^2}\right).
$$
\end{lemma}

The third lemma bounds the remaining Fourier coefficients which are near
$M$ but far from $L$. In the following let $M \setminus L$ denote the set
of elements in $M$ but not in $L$.

\begin{lemma}[Bounding Fourier coefficients near $M$ but far from $L$]
\label{lem:bound_fourier_near_M_far_L}
Assume the conditions of Theorem~\ref{thm:main} and fix $\eps \le 1/(2
c_1 m)$. Let $\theta \in \T^A$ be such that $d(\theta,M \setminus L)
\le \eps$. Then
$$
|\hat{X}(\theta)| \le
    \exp\left(-N \cdot \frac{O(1)}{m^2 |A| \log(c_1 |A|)}\right).
$$
\end{lemma}

We prove
Lemmas~\ref{lem:estimate_fourier_near_zero},~\ref{lem:bound_fourier_far_M}
and~\ref{lem:bound_fourier_near_M_far_L} in
Sections~\ref{sec:fourier_near_zero},~\ref{sec:bound_fourier_far_M}
and~\ref{sec:bound_fourier_near_M_far_L}, respectively. We
combine them to prove Theorem~\ref{thm:main} in
Section~\ref{sec:proof_of_main_thm_from_lemmas}.

\subsection{Estimating Fourier coefficients near zero}
\label{sec:fourier_near_zero}

Let $\theta \in \T^A$ be such that $\|\theta\|_2 \le \eps$. We may assume
that $\eps \le O(1/(c_1 N^{1/3}))$ otherwise the conclusion of the lemma
is trivial. We decompose
\begin{equation}\label{eq:estimate_fourier_decompose}
e^{-2 \pi i \ip{\Ex[X], \theta}} \cdot \hat{X}(\theta)
    = \prod_{b \in B} \left(e^{-2 \pi i \cdot p \ip{\phi(b),\theta}}
    \cdot (1-p + p e^{2\pi i \ip{\phi(b),\theta}})\right).
\end{equation}
Let $\nu_b:=\ip{\phi(b),\theta}$ where the inner product is taken over
$\R$. Since we assume $\|\theta\|_2 \le \eps$ we can bound $|\nu_b| \le
\|\phi(b)\|_2 \|\theta\|_2 \le c_1 \eps \ll 1$. Thus we can approximate the
terms in~\eqref{eq:estimate_fourier_decompose} by their Taylor series. The
following claim gives a cubic approximation.

\begin{claim}\label{cl:taylor_approx_single} Let $f:\R \to \C$ be given
by $f(x):=e^{-ipx} (1-p + p e^{ix})$. Then for $|x| \le 1$ we have
$$
f(x) = e^{-p x^2}(1+\delta),
$$
where $|\delta| \le O(p^2 x^2 + p x^3)$.
\end{claim}

\begin{proof}
We compute the cubic approximation for $f(x)$ as a polynomial in $p,x$. In
the following we use shorthand expression $x=y+O(z)$ for $|x-y|=O(z)$. We
have
\begin{align*}
f(x) &= (1-p)e^{-ipx}+pe^{i(1-p)x}\\
&=(1-p)(1-ipx + O(p^2 x^2)) + p(1+i(1-p)x-x^2\pm O(p x^2+x^3))\\
&=1-px^2 + O(p^2 x^2 + px^3)\\
&=e^{-p x^2} + O(p^2 x^2 + p x^3).\qedhere
\end{align*}
\end{proof}

We next apply the approximation given in
Claim~\ref{cl:taylor_approx_single} to each of the terms appearing
in~\eqref{eq:estimate_fourier_decompose}. Summing up the errors, and using
the fact that each term is bounded in absolute value by $1$, we get that
\begin{equation}\label{eq:estimate_fourier_approx}
\hat{X}(\theta) = e^{2 \pi i \ip{\Ex[X],\theta}}
    e^{-4 \pi^2 p \cdot \sum_{b \in B} \nu_b^2} (1+\delta)
\end{equation}
where $|\delta| \le O(p^2 \sum_{b \in B} \nu_b^2 + p \sum_{b \in B} \nu_b^3)$.
To conclude the proof, note that
$$
\sum_{b \in B} \nu_b^2 = \sum_{b \in B} \ip{\phi(b),\theta}^2
    = \theta^T R \theta,
$$
where we recall that $R_{a',a''} = \sum_{b \in B} \phi(b)_{a'}
\phi(b)_{a''}$. To bound the error term, recall that $|\nu_b| \le c_1 \eps
\ll 1$ hence
$$
|\delta| \le O(p^2 |B| + p |B| (c_1 \eps)^3) = O(N^2/|B|+N c_1^3 \eps^3).
$$

\subsection{Bounding Fourier coefficients far from $M$}
\label{sec:bound_fourier_far_M}

Let $\theta \in \T^A$ be such that $d(\theta,M) \ge \eps$. Thus, there exists
at least on coordinate $\theta_a$ whose distance from multiples of $1/m$ is
at least $\eps/\sqrt{|A|}$. Otherwise put, there exists $a \in A$ such that
\begin{equation}\label{eq:lower_bound_m_times_theta_a}
| m \theta_a \pmod{1}| \ge \eps m / \sqrt{|A|}.
\end{equation}
Recall that the Fourier coefficient $\hat{X}(\theta)$ is given by
$$
\hat{X}(\theta) = \prod_{b \in B} (1-p+p e^{2 \pi i \ip{\phi(b),\theta}}).
$$
Hence, to get a bound on $|\hat{X}(\theta)|$  essentially we need
to show that $\ip{\phi(b),\theta}$ is far from integer for many
$b \in B$. Note that we cannot longer assume, as in the proof of
Lemma~\ref{lem:estimate_fourier_near_zero}, that $\ip{\phi(b),\theta}$
is small in absolute value, since we assume no upper bound on
$\|\theta\|_2$. Thus, it may be the case that $\ip{\phi(b),\theta}$ is
large but still approximately integer. Let $\nu_b := \ip{\phi(b),\theta}
\pmod{1}$ where $|\nu_b| \le 1/2$. Our goal is to show that $|\nu_b|$
is noticeably large for many values $b \in B$. This will then imply the
required upper bound on $|\hat{X}(\theta)|$.

We will show this using the isolation vectors guaranteed by
Theorem~\ref{thm:main}. Let $\gamma \in \Z^B$ be an isolation vector for
$a$ with modulus $m$; that is $\phi(\gamma) = m \cdot e_a$. We first show
that it cannot be that $\nu_b \approx 0$ for all $b \in \supp(\gamma)$.

\begin{claim}\label{cl:one_isolation_vector_bound}
Let $\gamma \in \Z^B$ be such that $\phi(\gamma) = m \cdot e_a$. Then
$$
\sum_{b \in \supp(\gamma)} |\nu_b|^2
    \ge \frac{\eps^2 m^2}{|A| \|\gamma\|_2^2}.
$$
\end{claim}

\begin{proof}
Using the isolation property of $\gamma$ we get that
\begin{align*}
\sum_{b \in \supp(\gamma)} \gamma_b \nu_b \pmod{1}
    &= \sum_{b \in \supp(\gamma)} \gamma_b \ip{\phi(b),\theta} \pmod{1}\\
    &= \ip{\phi(\gamma),\theta} \pmod{1} = m \theta_a \pmod{1}.
\end{align*}
Hence by~\eqref{eq:lower_bound_m_times_theta_a} we get that $|\sum_{b \in
\supp(\gamma)} \gamma_b \nu_b \pmod{1}| \ge \eps m /\sqrt{|A|}$. On the
other hand, we can bound
$$
|\sum_{b \in \supp(\gamma)} \gamma_b \nu_b \pmod{1}|
    \le |\sum_{b \in \supp(\gamma)} \gamma_b \nu_b|
    \le \|\gamma\|_2 \sqrt{\sum_{b \in \supp(\gamma)} |\nu_b|^2}.
$$
Combining the two bounds, we get that $\sum_{b \in \supp{\gamma}} |\nu_b|^2
\ge \eps^2 m^2 / |A| \|\gamma\|_2^2$ as claimed.
\end{proof}

We now use the assumption of Theorem~\ref{thm:main} on the existence
of many vectors which isolate $a$ with disjoint support. Recall that
by assumption we have $r \ge |B|/c_2$ vectors $\gamma_1,\ldots,\gamma_r
\in \Z^B$ such that: (1) each $\gamma_i$ isolates $a$ with modulus $m$;
(2) The vectors $\gamma_1,\ldots,\gamma_r$ have disjoint supports;
and (3) $\|\gamma_i\| \le c_3$ for all $i \in [r]$. Applying
Claim~\ref{cl:one_isolation_vector_bound} to each vector $\gamma_i$
independently we derive that
\begin{equation}\label{eq:isolation_bound_many}
\sum_{b \in B} |\nu_b|^2 \ge \eps^2 |B|\cdot \frac{m^2}{|A| c_2 c_3^2}.
\end{equation}

To conclude the proof of the lemma, we apply~\eqref{eq:isolation_bound_many}
to derive an upper bound on $|\hat{X}(\theta)|$. The following claim
is simple.
\begin{claim}\label{cl:bound_by_exp}
Let $p \le 1/2$ and $|x| \le 1/2$. Then
$$
|1-p + p e^{2 \pi i x}| \le \exp(-p x^2).
$$
\end{claim}
Applying Claim~\ref{cl:bound_by_exp} we derive the bound
$$
|\hat{X}(\theta)| = \prod_{b \in B} |1-p+p e^{2 \pi i \cdot \nu_b}|
    \le \exp\left(-p \sum_{b \in B} |\nu_b|^2\right)
    \le \exp\left(-\eps^2 N \cdot \frac{c m^2}{|A| c_2 c_3^2}\right).
$$

\subsection{Bounding Fourier coefficients near $M$ but far from $L$}
\label{sec:bound_fourier_near_M_far_L}

Let $\theta \in \T^A$ be such that $d(\theta,M \setminus L) \le \eps$. That
is, there exists $\alpha \in M \setminus L$ such that $d(\theta,\alpha)
\le \eps$. Since $\alpha \notin L$ there must exist $b^* \in B$ such that
$\ip{\phi(b^*),\alpha} \notin \Z$.  We will show using the symmetry of $\phi$
that in fact this holds for many $b \in B$.  Moreover, since $\alpha \in M$
we have that if $\ip{\phi(b),\alpha} \notin \Z$ is must be at least $1/m$
far from the integers. This will allow us to give strong upper bounds
on the Fourier coefficient $\hat{X}(\alpha)$ and by continuity also on
$\hat{X}(\theta)$.

Let $\mathcal{L}$ denote the lattice generated by $\{\phi(b): b \in B\}$. In
other words, $\mathcal{L}$ is the subgroup of $\Z^A$ whose elements are
all possible integer combinations of $\{\phi(b): b \in B\}$. We first
show that any subset of $B$ which generates the lattice $\mathcal{L}$
must contain $b$ for which $\ip{\phi(b), m \alpha} \ne 0$.

\begin{claim}
\label{cl:spanning_set_contains_b}
Let $K \subset B$ be a set which generates the lattice $\mathcal{L}$. Then
there must exist $b \in K$ for which $\ip{\phi(b), m\alpha} \ne 0$.
\end{claim}
\begin{proof}
By assumption since $K$ generates the lattice $\mathcal{L}$, we can express
$\phi(b^*)$ as an integer combination of $\{\phi(b): b \in K\}$. That is,
there exist integer coefficient $\alpha_b$ for $b \in K$ such that
$$
\phi(b^*) = \sum_{b \in K} \alpha_b \phi(b).
$$
Thus, as $\ip{\phi(b^*), m \alpha} \ne 0$, there must exist $b \in K$
for which $\ip{\phi(b), m \alpha} \ne 0$ as well.
\end{proof}

We next claim that there must exist at least one small set $K \subset B$
which generates $\mathcal{L}$. We will later use symmetry to generate from
it many such sets.

\begin{claim} \label{cl:small_spanning_set_exists}
There exists $K \subset B$ of size $|K| \le O(|A| \log(c_1 |A|))$ such
that $\{\phi(b): b \in K\}$ generates the lattice $\mathcal{L}$.
\end{claim}

\begin{proof}
Let $K$ be a minimal subset of $B$ such that $\{\phi(b): b \in K\}$
generates the lattice $\mathcal{L}$. We claim that the minimality of $K$
implies that all partial sums $\phi(K')$ for $K' \subseteq K$ must be
distinct. Otherwise, assume that there exist two distinct subsets $K_1,K_2
\subseteq K$ for which $\phi(K_1)=\phi(K_2)$. We can assume w.l.o.g that
$K_1,K_2$ are disjoint by removing common elements from both. Thus we have
$$
\sum_{b \in K_1} \phi(b) - \sum_{b \in K_2} \phi(b)=0.
$$
In particular, we can express any $b' \in K_1 \cup K_2$ as an integer
combination of $\{\phi(b): b \in K \setminus \{b'\}\}$. Thus, we can remove
$b'$ from $K$ and maintain the property that the resulting set generates
$\mathcal{L}$. This contradicts the minimality of $K$.

We thus know that all sums $\{\phi(K'): K' \subseteq K\}$ are distinct. We
now apply the assumption that $\phi$ is bounded. By the assumptions of
Theorem~\ref{thm:main} we know that $\|\phi(b)\|_{\infty} \le \|\phi(b)\|_2
\le c_1$. Hence we conclude that
$$
\{\phi(K'): K' \subseteq K\} \subseteq [-c_1 K,c_1 K]^A,
$$
which imply that
$$
2^K \le (2c_1 K+1)^{|A|}.
$$
It is easy to verify that this gives the bound $K \le O(|A| \log(c_1 |A|))$
as claimed.
\end{proof}

The next step is to use the symmetry of $\phi$ to generate many small sets
which span $\mathcal{L}$.

\begin{claim}
\label{cl:small_spanning_set_symmetry}
Let $K \subset B$ be a set such that $\{\phi(b): b \in K\}$ generates
the lattice $\mathcal{L}$. Let $(\pi,\tau) \in S_B \times \GL(V)$ be a
symmetry of $\phi$. Let $K_{\pi} := \{\pi(b): b \in K\}$ be a shift of $K$
by $\pi$. Then $\{\phi(b): b \in K_{\pi}\}$ also generates the lattice
$\mathcal{L}$.
\end{claim}

\begin{proof}
Let $b' \in B$. We need to show that we can express $\phi(b')$ as integer
combination of $\{\phi(\pi(b)): b \in K\}$.  Consider $\pi^{-1}(b')$. By
assumption the image of $\phi$ on elements of $K$ generates the lattice
$\mathcal{L}$, hence there exist coefficients $\alpha_b \in \Z$ for $b
\in K$ such that
$$
\phi(\pi^{-1}(b')) = \sum_{b \in K} \alpha_b \phi(b).
$$
Applying the assumption that $(\pi,\tau)$ is a symmetry of $\phi$ we get that
$$
\phi(b')=\phi(\pi(\pi^{-1}(b')))=\tau(\phi(\pi^{-1}(b'))) = \sum_{b
\in K} \alpha_b \cdot \tau(\phi(b)) = \sum_{b \in K} \alpha_b
\phi(\pi(b)).\qedhere
$$
\end{proof}

We combine Claims~\ref{cl:spanning_set_contains_b},
\ref{cl:small_spanning_set_exists} and
\ref{cl:small_spanning_set_symmetry} to derive that $\ip{\phi(b), \alpha}
\ne 0$ for many $b \in B$.  Let $\widetilde{B}=\{b \in B: \ip{\phi(b),
\alpha} \ne 0\}$.

\begin{cor}\label{cor:many_b_not_orth_alpha}
$|\widetilde{B}| \ge \Omega \left(\frac{|B|}{|A| \log(c_1 |A|)}\right)$.
\end{cor}

\begin{proof}
Let $K$ be the set guaranteed by Claim~\ref{cl:small_spanning_set_exists}
where $|K| \le O(|A| \log(c_1 |A|))$.  Let $K_{\pi}=\{\pi(b): b \in K\}$. We
know by Claim~\ref{cl:small_spanning_set_symmetry} that for any symmetry
$(\pi,\tau)$ of $\phi$ we have
$$
|K_{\pi} \cap \widetilde{B}| \ge 1.
$$
Let $H$ be the subgroup of permutations on $B$ given by symmetries of
$\phi$. That is, $H=\{\pi: (\pi,\tau) \textrm{ symmetry of } \phi\}$. We
know by the assumptions of Theorem~\ref{thm:main} that $H$ acts transitively
on $B$. Thus, for any fixed $b \in B$, if we choose $\pi \in H$ uniformly
we have that $\pi(b)$ is uniformly distributed in $B$. Thus,
$$
\Ex_{\pi \in H}[K_{\pi} \cap \widetilde{B}]
    = \sum_{b \in K} \Pr_{\pi \in H}[\pi(b) \in \widetilde{B}]
    = \frac{|K||\widetilde{B}|}{|B|}.
$$
We thus conclude that we must have $|\widetilde{B}| \ge |B|/|K|$.
\end{proof}

We conclude the proof of Lemma~\ref{lem:bound_fourier_near_M_far_L}
by establishing an upper bound of $\hat{X}(\alpha)$. For any $b \in
\widetilde{B}$ we have that $\ip{\phi(b),\alpha} \pmod{1} \ne 0$, hence
since $\alpha \in (1/m \cdot \Z)^A$ we have
$$
|\ip{\phi(b),\alpha} \pmod{1}| \ge 1/m.
$$
Recall that by assumption $\|\phi(b)\|_2 \le c_1$ and $\|\alpha-\theta\| \le
1/(2 c_1 m)$. Thus $|\ip{\phi(b),\alpha-\theta}| \le 1/2m$ by Cauchy-Schwarz
and we get that
$$
|\ip{\phi(b),\theta} \pmod{1}| \ge 1/2m.
$$
We thus conclude with an upper bound on $|\hat{X}(\theta)|$. Applying
Claim~\ref{cl:bound_by_exp} we have
\begin{align*}
|\hat{X}(\theta)| \le \prod_{b \in \widetilde{B}}|1-p+p e^{2 \pi i
\ip{\phi(b),\theta}}| \le \exp(-p (1/2m)^2 |\widetilde{B}|) \le \exp\left(-N
\cdot \frac{O(1)}{m^2 |A| \log(c_1 |A|)}\right).
\end{align*}

\subsection{Proof of Theorem~\ref{thm:main} from
    Lemmas~\ref{lem:estimate_fourier_near_zero},
    \ref{lem:bound_fourier_far_M} and~\ref{lem:bound_fourier_near_M_far_L}}
\label{sec:proof_of_main_thm_from_lemmas}

We now deduce Theorem~\ref{thm:main} from
Lemmas~\ref{lem:estimate_fourier_near_zero},~\ref{lem:bound_fourier_far_M}
and~\ref{lem:bound_fourier_near_M_far_L}. Recall that we have
\begin{equation}\label{eq:pr_fourier_again}
\Pr[X=\Ex[X]] = \int_{\theta \in \T^A} \hat{X}(\theta)
    e^{-2 \pi i \ip{\Ex[X],\theta}} d \theta.
\end{equation}
Let $N=\poly(|A|,m,c_0,c_1,c_2,c_3)$ large enough to be chosen
later. We would assume throughout that $N$ is a multiple of $c_0
m$.  If $|B| = O(N^2)$ then the set $B$ is small to begin with, so
assume that $|B| \gg N^2$.  We set $\eps \approx N^{-1/3}$ so that
the conditions for Lemmas~\ref{lem:estimate_fourier_near_zero}
and~\ref{lem:bound_fourier_near_M_far_L} hold. More explicitly,
we set $\eps:=O(1/ c_1 N^{1/3})$ so that the conditions for
Lemma~\ref{lem:estimate_fourier_near_zero} hold with $|\delta| \le 1/2$;
and we assume that $N \ge \Omega(m^3)$ so that $\eps \le 1/(2 c_1 m)$ and
the conditions for Lemma~\ref{lem:bound_fourier_near_M_far_L} also hold.

We decompose the integral in~\eqref{eq:pr_fourier_again} into three
integrals: over points which are $\eps$ close to $L$; over points which
are $\eps$ close to $M \setminus L$; and over points which are $\eps$
far from $M$. Our choice of $\eps < 1/2m$ also guarantees that balls of
radius $\eps$ around distinct points in $M$ are disjoint. We thus have
that $\Pr[X=\Ex[X]]=I_1+I_2+I_3$ where
\begin{align*}
I_1 &:= \sum_{\alpha \in L} \int_{\theta \in \T^A: d(\theta,\alpha) \le \eps}
    \hat{X}(\theta) e^{-2 \pi i \ip{\Ex[X],\theta}} d \theta,\\
I_2 &:= \sum_{\alpha \in M \setminus L} \int_{\theta \in \T^A: d(\theta,\alpha)
    \le \eps} \hat{X}(\theta) e^{-2 \pi i \ip{\Ex[X],\theta}} d \theta,\\
I_3 &:= \int_{\theta \in \T^A: d(\theta,M) > \eps} \hat{X}(\theta)
    e^{-2 \pi i \ip{\Ex[X],\theta}} d \theta.
\end{align*}

We first lower bound $I_1$.

\begin{claim}\label{cl:I1}
$$
I_1 \ge |L| \left(\frac{\Omega(1)}{c_1 N^{1/2} |A|^{1/2}} \right)^{|A|}.
$$
\end{claim}

\begin{proof}
We first use the assumption that $N$ divides $c_0 m$ to reduce computing
$I_1$ to an integral around $0$. We claim that the assumption that $c_0 m |
N$ implies that $\ip{\Ex[X],\alpha} \in \Z$ for all $\alpha \in L$. This
is since this choice implies that all entries of $\Ex[X]$ are divisible
by $m$ since
$$
\Ex[X] = \frac{N}{|B|} \phi(B) = (N/c_0 m) \cdot m \cdot \frac{c_0}{|B|}
    \phi(B) \in m \Z^A.
$$
Moreover, since $\alpha \in L \subset M$ we have that $m \alpha
\in \Z^A$, hence $\ip{\Ex[X],\alpha} \in \Z$. Combining this with
Claim~\ref{cl:properties_L} which states that the Fourier coefficients
of $X$ are invariants to shifts by $\alpha \in L$, we deduce that
$$
I_1 = |L| \int_{\theta \in \T^A: \|\theta\|_2 \le \eps}
    \hat{X}(\theta) e^{-2 \pi i \ip{\Ex[X],\theta}} d \theta.
$$
Recall that by Lemma~\ref{lem:estimate_fourier_near_zero} and our choice
of parameters, if $\|\theta\|_2 \le \eps$ then
$$
\hat{X}(\theta) = \widetilde{X}(\theta)(1+\delta(\theta))
$$
where $\widetilde{X}(\theta)=e^{2 \pi i \ip{\Ex[X],\theta}} e^{-4 \pi^2
p \cdot \theta^T R \theta}$ and where $|\delta(\theta)| \le 1/2$. Hence
$$
I_1 = |L| \int_{\theta \in \T^A: \|\theta\|_2 \le \eps}
    e^{-4 \pi^2 p \cdot \theta^T R \theta} (1+\delta) d \theta.
$$
Consider
\begin{align*}
I'_1 = |L| \int_{\theta \in \T^A: \|\theta\|_2 \le \eps}
    e^{-4 \pi^2 p \cdot \theta^T R \theta} d \theta.
\end{align*}
We claim that $|I_1| \ge |I'_1|/2$, hence it suffices to lower bound
$|I'_1|$ in order to lower bound $|I_1|$. To see that, note that $I'_1$
is an integral of a real positive function; that we can always lower bound
$|I_1|$ by its real part $\Re(I_1)$; and that $\Re(1+\delta) \ge 1/2$
since $|\delta| \le 1/2$. Thus
$$
|I_1| \ge \Re(I_1) \ge \Re(I'_1)/2 = I'_1/2.
$$
We next lower bound $I'_1$. Note first that we can bound $\theta^T R \theta
\le B c_1^2 \|\theta\|_2^2$. This is because
$$
\theta^T R \theta = \sum_{b \in B} \ip{\theta,\phi(b)}^2
    \le \sum_{b \in B} \|\theta\|_2^2 \|\phi(b)\|_2^2
    \le B c_1^2 \|\theta\|_2^2.
$$
Thus we get that
$$
I'_1 \ge |L| \int_{\theta \in \T^A: \|\theta\|_2 \le \eps} e^{-4 \pi^2
c_1^2 N  \|\theta\|_2^2} d \theta.
$$
We bound $I'_1$ from below by the volume of the region in which the
integrand is constant. This occurs whenever $\|\theta\|_2 \le \eps'=O(1/(c_1
N^{1/2}))$. Recall that we chose $\eps = O(1/(c_1 N^{1/3})) \gg \eps'$. Hence
the ball of radius $\eps'$ is contained in the area over which we integrate,
so we obtain the lower bound
$$
|I_1| \ge I'_1/2 \ge |L| \cdot O(1) \cdot
\mathrm{Vol}(\mathrm{Ball}(0,\eps')) = L \cdot
\left(\frac{\Omega(1)}{\eps' |A|^{1/2}} \right)^{|A|}= L \cdot
\left(\frac{\Omega(1)}{c_1 N^{1/2} |A|^{1/2}} \right)^{|A|}.\qedhere
$$
\end{proof}

The next steps are to bound $I_2$ and $I_3$ from above. We bound them by the
maximal value that $|\hat{X}(\theta)|$ can achieve in their integral
domains. Lemma~\ref{lem:bound_fourier_near_M_far_L} gives a bound on $I_2$,
$$
|I_2| \le \max\{|\hat{X}(\theta)|: d(\theta,M \setminus L) \le \eps\}
    \le \exp\left(-N \cdot \frac{O(1)}{m^2 |A| \log(c_1 |A|)}\right),
$$
and Lemma~\ref{lem:bound_fourier_far_M} and our choice of $\eps =
O(1/(c_1 N^{1/3}))$ gives a bound on $I_3$,
$$
|I_3| \le \max\{|\hat{X}(\theta)|: d(\theta,M) \ge \eps\} \le
    \exp\left(-N \eps^2 \cdot \frac{m^2}{|A| c_2 c_3^2}\right)=
    \exp\left(-O(N^{1/3}) \cdot \frac{m^2}{|A| c_1^2 c_2 c_3^2}\right).
$$

We now need to choose $N$ large enough so that $I_1 \gg |I_2|,|I_3|$. This
can be accomplished since $I_1$ decays polynomially with $N$, while
$|I_2|,|I_3|$ decay exponentially fast. It is not hard to verify that this
is guaranteed whenever
$$
N \ge \Omega(1) \cdot \max(A^2 m^2 \log^2(m A c_0 c_1 c_2 c_3),
    A^6 c_1^6 c_2^3 c_3^6 \log^3(m A c_0 c_1 c_2 c_3)).
$$

\section{Summary and open problems}
\label{sec:summary}

Our main theorem guarantees the existence of a small subset $T \subset
B$ for which \eqref{eq:prob_def} holds. The conditions we require are
boundedness, divisibility, symmetry and isolation. The first three
conditions seem natural for this type of problems, but the fourth seems
artificial, as it depends on the specific basis we choose for $V$. Thus,
we wonder if this condition can be removed. In particular, the following
question captures much of the difficulty.  Let $G$ be a group that acts
transitively on a set $X$. A subset $T \subset G$ is \emph{$X$-uniform}
(or an \emph{$X$-design}) if it acts on $X$ exactly as $G$ does. That is,
for any $x,y \in X$,
$$
\frac{1}{|T|} |\{g \in T: g(x)=y\}| = \frac{1}{|G|} |\{g \in G:
g(x)=y\}| = \frac{1}{|X|}.
$$
In our language we may take $B=G$ and $V$ to be the space
spanned by all functions $\phi_{(x,y)}:B\to\{0,1\}$ of the form
$\phi_{(x,y)}(b)=\mathbf{1}_{\{b(x)=y\}}$ for $x,y\in X$. Then $T$ is
$X$-uniform if and only if \eqref{eq:prob_def} holds. Taking $A$ to be
some subset of $X\times X$ for which $(\phi_a)_{a\in A}$ forms a basis of
$V$, the boundedness, divisibility and symmetry conditions are clearly
satisfied. However, it is not clear whether the isolation condition is
satisfied as well.  If indeed the isolation condition is redundant, one
may conjecture that:

\begin{conj} Let $G$ be a group that acts transitively on a set $X$. Then
there exists an $X$-uniform subset $T \subset G$ such that $|T| \le |X|^c$
for some universal constant $c>0$.
\end{conj}

A second question is whether one can apply our techniques to get
\emph{minimal} objects. Recall that the size of the objects we achieve
is only minimal up to polynomial factors. For example, one of the main
open problems in design theory is whether there exists a Steiner system
(i.e. a $t$-design with $\lambda=1$) for any $t > 5$.  Another major
open problem of a similar spirit is the existence of Hadamard matrices
of all orders $n=4m$, or equivalently, $2$-$(4m-1, 2m-1, m-1)$ designs.
Empirical estimates for $n \le 32$ suggest that there are $\exp(O(n(\log
n)))$ Hadamard matrices of order $n = 4m$.  Since are so many of them,
and since the logarithm of their number grows at a regular rate, we suspect
that they exist for some purely statistical reason.  However, the Gaussian
local limit model seems to be false for Hadamard matrices interpreted as
$t$-designs; it does not accurately estimate how many there are.

A third question is whether there exists an algorithmic version
of our work, similar to the algorithmic Moser~\cite{Moser2009} and
Moser-Tardos~\cite{MoserTardos2010} versions of the Lov\'{a}sz local
lemma~\cite{ErdosLovasz73}, and the algorithmic Bansal~\cite{Bansal2010}
version of the six standard deviations method of Spencer~\cite{Spencer1985}.
If an efficient randomized algorithm of our method were found, then we
could no longer indisputably claim that we have a low-probability version
of the probabilistic method.  On the other hand it would be strange,
from the viewpoint of computational complexity theory, if low-probability
existence can always be converted to high-probability existence.  Maybe our
construction is fundamentally a low-probability construction.

\bibliographystyle{hamsalpha}
\bibliography{designs}

\end{document}